\documentclass[11pt]{amsart}

\usepackage{epigamath}

\usepackage[english]{babel}

\numberwithin{equation}{section}

\usepackage{enumitem}
\usepackage[all]{xy}
\usepackage{tikz-cd}

\newtheorem{bigthm}{Theorem}
 
\newtheorem{thm}{Theorem}[section]
\newtheorem{lem}[thm]{Lemma}
\newtheorem{prop}[thm]{Proposition}

\theoremstyle{definition}

\newtheorem{defn}[thm]{Definition}

\theoremstyle{remark}

\newtheorem*{example}{Example}
\newtheorem*{rmk}{Remark}

\newcommand{\sC}{\mathcal{C}}

\newcommand{\sO}{\mathcal{O}}

\newcommand{\sQ}{\mathcal{Q}}

\newcommand{\sS}{\mathcal{S}}

\newcommand{\mC}{\mathbb{C}}

\newcommand{\mP}{\mathbb{P}}
\newcommand{\mQ}{\mathbb{Q}}

\newcommand{\To}{\longrightarrow}

\EpigaVolumeYear{7}{2023} \EpigaArticleNr{3}
\ReceivedOn{January 20, 2022}
\InFinalFormOn{September 5, 2022}
\AcceptedOn{October 6, 2022}

\title{Quartic surfaces, their bitangents and rational points}
\titlemark{Quartic surfaces, their bitangents and rational points}

\author{Pietro~Corvaja}
\address{Dipartimento di Scienze Matematiche, Informatiche e Fisiche, Universit\`a di Udine, Udine, 33100, Italy}
\email{pietro.corvaja@uniud.it}

\author{Francesco Zucconi}
\address{Dipartimento di Scienze Matematiche, Informatiche e Fisiche, Universit\`a di Udine, Udine, 33100, Italy}
\email{francesco.zucconi@uniud.it}

\authormark{P.~Corvaja and F.~Zucconi}

\AbstractInEnglish{Let $X$ be a smooth quartic surface not containing lines, defined over a number field $\kappa$. We prove that there are only finitely many bitangents to $X$  which are defined over $\kappa$. This result can be interpreted as saying that a certain surface, having vanishing irregularity, contains only finitely many rational points. In our proof, we use the geometry of lines of the quartic double solid associated to $X$. 

In a somewhat opposite direction, we show that on any quartic surface $X$ over a number field $\kappa$, the set of algebraic points in $X(\bar{\kappa})$ which are quadratic over a suitable finite extension $\kappa'$ of $\kappa$ is Zariski-dense.}

\MSCclass{14G0599, 11D99}

\KeyWords{Quartic surfaces, rational points}

\begin{document}



\maketitle

\begin{prelims}

\DisplayAbstractInEnglish

\bigskip

\DisplayKeyWords

\medskip

\DisplayMSCclass







\end{prelims}


\newpage

\setcounter{tocdepth}{1}

\tableofcontents


\section{Introduction}

\subsection{The result}

In the recent paper \cite{CoZu} we investigated the geometry of the surface $S$ which parametrizes the bitangents to a generic smooth quartic surface in $\mP^3$.
Although our results in \cite{CoZu} were of a geometric nature, and had been formulated over the complex number field, one of the motivation of our research came from arithmetic. In the present paper we develop, among other things, some arithmetic consequences of our previous researches.

Given a smooth quartic surface $X\subset \mP^3$ defined over a number field $\kappa$, it is widely believed that its set of rational points is potentially dense, by this we mean that there should exist a number field $\kappa'\supset\kappa$ such that the set $X(\kappa')$ is Zariski dense. For instance, $K3$ surfaces are ``special varieties'' in the sense of Campana (see \cite{Campana}).
While this is proved only in special cases, all of which requiring the N\'eron Severi group of  $X$  having rank bigger than one (see \emph{e.g.} the papers of Bogomolov, Harris, Hasset and Tschinkel \cite{Bom-Tsch,Harris-Tsch,Has-Tsch} on elliptic $K3$ surfaces, and the older paper \cite{SD} by Swinnerton-Dyer on the Fermat quartic surface) we can prove the following weaker (and elementary) result for general quartic surfaces.

\begin{bigthm}\label{Th:A}
Let $X$ be a smooth quartic surface over a number field $\kappa$. There exists a finite extension $\kappa'/\kappa$ such that the set of algebraic points in $X(\bar{\kappa})$ which are quadratic over $\kappa'$ is Zariski-dense.
\end{bigthm}

Theorem~\ref{Th:A} is proved by showing that every quartic surface can be covered by hyperelliptic curves of genus two, all defined over a fixed number field together with their canonical involution, such that the corresponding genus zero quotient has a rational point.  Alternatively, one could use two genus one curves on $X$: over a sufficiently large number field they would provide infinite families of secants defined over the same number field, each providing two quadratic points on $X$. 

We note that the irrationality degree of a general quartic surface is three (\emph{cf.} \cite{Yosh}). This implies that there are no degree two rational maps $X \dashrightarrow \mP^2$, so that the quadratic rational points on $X$ cannot be constructed in the obvious way as pre-images of rational points on the plane. Also, since the irregularity of $X$ is zero, there cannot exist degree two rational maps $X\to A$ where $A$ is an Abelian surface (actually no dominant maps at all), so that again the quadratic rational points cannot be constructed from the rational points on an Abelian surface. 

\smallskip

Quadratic rational points can be naturally obtained from rational bitangents: every bitangent defined over $\kappa$ intersects the surface $X$ at two quadratic (exceptionally rational) points. Hence the presence of infinitely many rational bitangents would imply the infinitude of  quadratic rational points. However, for generic quartics we prove the finiteness of such rational bitangents.

\begin{bigthm}\label{Th:B}
Let $X\subset\mP^3$ be a smooth quartic surface, defined over a number field $\kappa$, containing no line. Then  there are only finitely many bitangent lines to $X$ which are defined over $\kappa$.
\end{bigthm}

Of course, this statement formally implies that  finiteness holds also for bitangents defined over a fixed finite extension of $\kappa$. Before going on, we remark in the projective space $\mP^{34}$ parametrizing quartic surfaces in $\mP^3$,  the $\mQ$-rational points  corresponding to smooth quartics with geometric Picard number one are Zariski-dense; this is the content of a  result of R. van   Luijk \cite{Luijk}, who also proves that one can construct such quartic surfaces so to contain infinitely many rational points. Since having geometric Picard number one prevents the existence of lines (over the algebraic closure of the field of definition), our theorem indeed admits concrete applications.

As mentioned at the beginning, the set of bitangents to $X$ is naturally endowed with the structure of an algebraic surface $S$; this surface is smooth if and only if $X$ contains no line.  Theorem~\ref{Th:B} consists in proving the finiteness of the rational points on such a surface $S$, over every number field, hence providing an instance of the validity in strong form of the so called Bombieri--Lang conjecture for surfaces of general type. The same proof also provides the hyperbolicity of such surfaces, both in the sense of Brody (absence of entire curves on $S$) and in the sense of Kobayashi (K. pseudo-distance being a real distance).

\smallskip

Besides the conjecture of Bombieri--Lang (for rational points) and Green--Griffiths (for entire curves), surfaces of general type are believed to contain only finitely many curves of geometric genus $\leq 1$. This is the content of the famous Bogomolov's conjecture. For such surfaces $S$,  we prove this conjecture in strong form.

\begin{bigthm}\label{Th:C}
Let $X$ be a smooth quartic, defined over the complex number field $\mC$, containing no line. Let $S$ be the surface parametrizing the bitangents to $X$. Then $S$ contains no curve of geometric genus $\leq 1$
\end{bigthm}

The proof of Theorem~\ref{Th:B} goes as follows. First by Chevalley--Weil Theorem (see \emph{e.g.} \cite{CZ-Hilb} for a discussion and alternative formulations) proving the finiteness of rational points on $S$ is tantamount to show it for the \'etale double cover $S_X\to S$ 
realised by the Hilbert scheme of lines $S_X$ of the quartic double solid $Q\to\mP^3$ branched on $X$. 
The geometry of $S$ and $S_X$ has been the object of thorough investigations by M. Welters \cite{W} and Tikhomirov \cite{T}, which will be used in the present paper, together with some results from our recent work \cite{CoZu}. In particular, by a theorem of Welters it follows that the irregularity of $S_X$ equals ten, in particular it exceeds the dimension of $S_X$. This is sufficient to deduce, \emph{via} a theorem of Faltings (see \cite[Theorem~1]{Fal}), the degeneracy of rational points on $S_X$. It remains to exclude possible infinite families contained in curves of genus $\leq 1$, and this is done \emph{via} our Theorem~\ref{Th:C}.

\smallskip

We end by noticing that the hypothesis on the Picard number of $X$ being equal to one cannot be omitted from Theorem~\ref{Th:B}. Indeed we found the following example of a smooth quartic admitting infinitely many bitangents  defined over $\mathbb{Q}$.

\begin{example}
Let $X\subset\mP^3$ be the smooth quartic surface of equation
\begin{equation*} 
x^4-xy^3=z^4-zw^3.
\end{equation*}
Then for every point $(s_0:s_1)\in\mP^1$, the line of equation
\begin{equation*}
\left\{
\begin{matrix}
s_0^3 x &=& s_1^3 z \\
s_0w &=& s_1 z
\end{matrix}
\right.
\end{equation*}
is quadritangent to $X$, in particular it is a bitangent.
\end{example}

Note that the surface $X$ defined above, known to some authors as  Schur's quartic surface,  contains $64$ lines (the maximal number of lines that can be contained in a smooth quartic) and its Picard number is $20$ (the maximal possible value for a quartic surface).
Geometrically, the above family of lines constitutes a rational curve on the corresponding surface $S$ parametrizing the bitangents to $X$; in this case, the surface $S$ is singular.

\section{The quartic surface}
 In this section we describe the geometry necessary to show Theorem~\ref{Th:C} and hence Theorem~\ref{Th:B}. 
 Let $V$ be a complex vector space of dimension $4$ and let $V^{\vee}$ be its $\mathbb C$-dual. We set $\mP^3:=\mP(V^\vee)$. Let $F\in {\rm{Sym}}^4V$ and let 
$X:=(F=0)\subset \mP^3$ be the associated quartic surface. 

We suppose from now on that $X$ is smooth, so that it is a K3 surface.
 
 \subsection{Morphism from elliptic curves to quartic surfaces} We first recall a proof of Theorem \ref{curel} stated below.
 We consider $X\subset\mP^3$ and the full linear system of hyperplane sections $\mP(H^0(X,\sO_{X}(1)))$. To impose two nodes counts for two independent algebraic conditions. 
 Hence the locus $\sC_{X}\hookrightarrow \mP(H^0(X,\sO_{X}(1)))$ given by those hyperplane sections with at least two nodes has dimension $\geq 1$.

\begin{thm}\label{curel} In any smooth quartic there exists a $1$-dimensional family of curves whose general member is irreducible of geometric genus $1$.
\end{thm}
\begin{proof} We consider two cases. If the general element of $\sC_X$ is irreducible then the general element has exactly two nodes (otherwise we would have an algebraic family of rational curves on a surface with vanishing irregularity; this, however, would imply that $X$ is rational). Hence we find a family of curves of geometric genus one.

Assume now that its general element is reducible. If this consists  into the union of two conics, then again we would have a one dimensional family of rational curves on $X$, which is impossible.

Then, it only remains to  consider the case where inside $\sC_X$ we have a linear system of hyperplane sections which split in the union of a fixed line and  a plane cubic. If the generic cubic in that family  were singular, $X$ would be again covered by an algebraic family of rational curves, which as we said is impossible. Then the generic cubic is smooth, \emph{i.e.} of genus one, and we have finished.
\end{proof}

\subsection{The surface of bitangents}
While Theorem~\ref{Th:A} holds for every smooth quartics, Theorem~\ref{Th:B} will be proved only for quartics which do  not contain any line (and our example in the introduction shows that this assumption cannot be completely removed).  Indeed we mostly rely on the work of Welters \cite{W} and the statements we use require that assumption to hold. So,  from now on {\it we assume that $X$  contains no line}. 

Let $\mathbb G:=\mathbb G(2, V^{\vee})$ be the Grassmann variety which parametrizes the lines of $\mP^3$. 
\begin{defn}\label{definizionedibitangente} A line $l\subset\mP^3$ is a bitangent line to $X$ if the subscheme $X_{|l}\hookrightarrow l$ is non reduced over each supporting point.
\end{defn}
Let $f_{l}$ be the binary form obtained by the restriction of $X=(F=0)$ to a line $l\subset\mP^3$.
\begin{defn} We call
\begin{equation}\label{superficiebitangenti}
S:=\left\{ [l]\in\mathbb G\mid X_{|l} =V(f_{l}) \ \text{where}\ f_{l}\ \text{is a binary biquadratic form}  \right\},
\end{equation}
{\it{the variety of bitangents to $X$}}.
\end{defn}

\subsection{The surface of contact points}
We have the standard exact sequence of vector bundles on $\mathbb G$:
\begin{equation}\label{universal-dual}
0\To \sQ^{\vee}\To V^{\vee}\otimes\sO_\mathbb G\To \sS\To 0.
\end{equation}
A point $\alpha\in \sQ^{\vee}$ is a couple $([l],v)$ where $[l]\in\mathbb G$, $v$ belongs to the $2$-dimensional vector space whose projectivisation is the projective line $l$. 

 We denote by $\mP(\sQ)$ the variety ${\rm{Proj}}({\rm{Sym}}(\sQ))$. Note for a vector space $V$, we denote by $\mP(V)$  the set of lines of $V$. By definition $\mP(\sQ)$ coincides with the universal family of lines over $\mathbb G$:
$$
\mP(\sQ)=\left\{ ([l],p) \in \mathbb G \times \mP^3 \mid p\in l \right\}.
$$
We denote by $\pi_{\mathbb G}\colon \mP(\sQ)\to\mathbb G$ the natural projection and following the mainstream we call {\it{the universal exact sequence}} the following one (dual to \eqref{universal-dual}):
\begin{equation}\label{universal}
0\To \sS^{\vee}\To V\otimes\sO_\mathbb G\To \sQ\To 0.
\end{equation}
If $H_{\mathbb G}$ is the hyperplane section of the Pl\"ucker embedding $\mathbb G\hookrightarrow \mP(\bigwedge^2V^{\vee})$, we stress that $\sO_\mathbb G(H_{\mathbb G})=\det\sQ=\det\sS$. Using the inclusion $j_{S}\colon S\hookrightarrow \mathbb G$ we can define $\sQ_{S}:=j_{S}^{\star}\sQ$. It remains to define the variety of contact points.
\begin{defn}\label{superficiebitangentipuntate}
We call
\begin{equation}
Y:=\left\{ ([l],p)\in S\times X\mid p \in X_{|l}  \right\}
\end{equation}
{\it{the variety of contact points}}.
\end{defn}
Obviously there is an embedding $j_Y\colon Y\hookrightarrow \mP(\sQ_{S})$ and the natural morphism $\pi_{S}\colon \mP(\sQ_{S})\to S$ restricts to a morphism $\pi\colon Y\to S$ which we call {\it{the forgetful morphism}}.

\subsubsection{Basic diagrams} Now we consider the standard conormal sequence of $X$ inside $\mP^3$:
\begin{equation}\label{conormal}
0\To \sO_X(-4)\To \Omega^1_{\mP^3|X}\To \Omega^1_X\To 0.
\end{equation}
\noindent
Following \cite{T} we can build the following diagram:
\begin{equation}\label{diagrammasullasuperficie}
\begin{tikzcd}
J_X\colon \mP\left(\Omega^1_X(1)\right)\ar[r]\ar[dr,"\rho_X"']&\mP\left(\Omega^1_{\mP^3|X}(1)\right)\ar[d,"\rho'"]& \mP\left(\sQ_S\right) \ar[dr,"\pi_S"]\ar[l]&\\
&X\subset\mP\left(V^\vee\right)&&S&
\end{tikzcd}
\end{equation}
where the inclusion $J_X\colon \mP\left(\Omega^1_X(1)\right)\to \mP\left(\Omega^1_{\mP^3|X}(1)\right)$ is given by the sequence \eqref{conormal} and the morphism 
$\mP(\sQ_S)\to \mP\left(\Omega^1_{\mP^3|X}(1)\right)$ is the restriction over $S$ of the standard diagram:
\begin{equation}\label{diagrammasullospazio}
\begin{tikzcd}
&\mP(\Omega^1_{\mP^3}(1))\cong\ar[dl,"\rho"']\mP(\sQ) \ar[dr,"\pi_{\mathbb G}"]&\\
\mP\left(V^\vee\right)&&  \mathbb G
\end{tikzcd}
\end{equation}
and $\rho'\colon \mP\left(\Omega^1_{\mP^3|X}(1)\right)\to \mP\left(V^\vee\right)$ is the obvious restriction.

\subsubsection{Geometrical interpretation}
The $\mP^1$-bundle $\pi_{\mathbb G}\colon \mP(\sQ)\to\mathbb G$ is the universal family of $\mathbb G$ by the sequence \eqref{universal} and the $\mP^2$-bundle $\rho\colon \mP\left(\Omega^1_{\mP^3}(1)\right)\to\mP^3$ is the projective bundle of the tangent directions on $\mP^3$, that is: $\rho^{-1}(p)=\mP\left(T_{\mP^3,p}\right)$ where $T_{\mP^3,p}$ is the vector space given by the tangent space to $\mP^3$ at the point $p$. The isomorphism $\mP\left(\Omega^1_{\mP^3}(1)\right)\cong \mP(\sQ)$ is well-known. 

It is also easy to verify that for the $3$-fold $\mP\left(\Omega^1_X(1)\right)$ we have:
$$\mP\left(\Omega^1_X(1)\right)\cong \left\{(p,[l])\in X\times \mathbb G\mid l\in \mP( T_{p} X) \right\}.$$
Inside $S$ there is the subscheme $B_{\rm{hf}}\hookrightarrow S$ which parametrizes the hyperflex lines. Using the surface of contact points $Y$ (\emph{cf.} Definition~\ref{superficiebitangentipuntate}) and $B_{\rm{hf}}$, we get an important fact about $S$.
\begin{prop}\label{ilrivestimento doppio ramificato}
There exists a non trivial $2$-torsion element $\sigma\in {\rm{Div}}(S)$ such that the surface of contact points $Y$ can be realised as a subscheme of $\mP\left(\sO_S\oplus\sO_S(\sigma+H_{\mathbb G|S})\right)$. The restriction of the natural projection $\mP\left(\sO_S\oplus\sO_S(\sigma+H_{|S})\right)\to S$ induces a $2$-to-$1$ cover $\pi\colon Y\to S$ branched over $B_{\rm{hf}}\in \left|2H_{\mathbb G|S}\right|$. In particular $Y$ is a smooth surface. 
\end{prop}
\begin{proof} See \cite[Proposition 3.11]{W}.
\end{proof}

By the diagram \eqref{diagrammasullasuperficie} and by a slight abuse of notation we obtain the basic diagram:
\begin{equation}
\label{oneuniversal&projection}
\begin{tikzcd}
& Y \ar[dl,"\rho"'] \ar[dr,"\pi"]&\\
X\subset\mP^{3} &  & S\subset\mathbb G 
\end{tikzcd}
\end{equation}
As in the papers \cite{W} and \cite{CoZu}, the geometry of the surface $Y$ will be crucial to understand that of $S$.

\section{The quartic double solid}
\subsection{Some general results about the notion of line of a quartic double solid} 

Let $\pi_Q\colon Q\to\mP^3$ be the $2$-to-$1$ cover branched over $X$. In \cite{T} and in \cite{W} it is shown how to relate the geometry of $Q$ with the one of $S$. We consider the tautological divisor $T_{\mP/\mP^3}$ of $\mP:=\mP\left(\sO_{\mP^3}\oplus \sO_{\mP^3}(2)\right)$, that is, following Grothendieck, $\rho_{\star}\sO_{\mP}(T_{\mP/\mP^3})=\sO_{\mP^3}\oplus \sO_{\mP^3}(2)$, where $\rho\colon \mP\to\mP^3$ is the natural projection. We denote by $H_{\mP^3}$ the hyperplane section of $\mP^3$ and we set 
$\sO_{\mP}(n):=\rho^{\star}\sO_{\mP^3}(nH_{\mP^3})$. If $T_1\in H^0(\mP, \sO_{\mP}(T_{\mP/\mP^3}))$ and 
$T_{\infty}\in H^0(\mP, \sO_{\mP}(T_{\mP/\mP^3}\otimes_{\mP}\sO_{\mP}(-2))$ it is easy to show that $Q\in \left|2T_{\mP/\mP^3}\right|$ and 
$$
Q= \left(T_1^2-F(x_0,x_1,x_2,x_3)T_{\infty}^2=0\right).
$$

\noindent By standard theory of double covers we get the following result.
\begin{lem}\label{invariantideltrifoglio} The Hodge numbers of $Q$ satisfy: $h^{i,j}(Q)=0$ if $i\neq j$, except $h^{1,2}(Q)=h^{2,1}(Q)=10$, $h^{i,i}(Q)=1$, $0\leq i\leq 3$. Moreover $\mathrm{Pic}(Q)=\mathbb Z$.
\end{lem}
\begin{proof} See \cite[p.~8]{W}.
\end{proof}

\subsection{Lines of the quartic double solid}
The threefold $Q$ is a Fano variety, that is the anti-canonical divisor $-K_{Q}\sim \rho^{\star}(2H_{\mP^3})$ is ample, and there is a natural notion of {\it{line}} of $Q$.

\begin{defn} A line of $Q$ is a connected subscheme $r\subset Q$ of pure dimension $1$ such that $r\cdot \rho^{\star}(H_{\mP^3})=1$.
\end{defn}

In \cite[p.~374]{T} there is a description of the lines of $Q$. Here we only recall that they come in couples of irreducible rational curves which mutually intersect into two points.

Thanks to the polarisation on $Q$ given by  $\rho^{\star}(H_{\mP^3})$ we can construct the Hilbert scheme of lines of $Q$. An important theorem by Iskovskikh (\emph{cf.} \cite{Is}) was used in \cite{T} to show the following statement.

\begin{prop}\label{liscia} The Hilbert scheme $S_X$  of lines of the quartic double solid is a $2$-to-$1$ \'etale cover $
f\colon S_X\to S$. 
\end{prop}
\begin{proof}  
The details of the proof are in \cite[Propositions~2.4 and~3.1]{T}. Another self-contained proof is in \cite[Lemma~1.1]{W}. See also \S~3 of \cite{CoZu}.
\end{proof}

\begin{thm}\label{formulae} For the surface $S_X$ we have the formulas:
\[K_{S_X}=f^*(3H_{\mathbb G|S}),\quad q(S_X)=10,\quad p_g(S_X)=101, \quad h^1(S_X,\Omega^1_{S_X}) =220,\quad \text{and}\quad c_2(S_{X})=384.\]
\end{thm}
\begin{proof} See \cite[Cohomological study, pp. 41-45]{W}.
\end{proof}

\subsection{Abelian varieties and quartic double solid}
Some properties of  the Abel--Jacobi map $S_X\to J(Q)$ where $J(Q)$ is the intermediate Jacobian of the quartic double solid $Q$ are well-known.
\begin{thm}
The Abel--Jacobi map ${\rm{Alb}} ( S_X)\to J(Q)$ is an isomorphism of Abelian varieties.
\end{thm}
\begin{proof}
See \cite[Theorem 4.1]{W}.
\end{proof}

\begin{thm}\label{collino} The differential of the Albanese map $\alpha\colon S_X\to {\rm{Alb}}(S_X)$ is injective at every point.
\end{thm}
\begin{proof} See \cite[Corollary A. 3]{W}.\end{proof}

\subsection{Special divisor inside the bitangents surface.}
For any $[l]\in S_X$ denote by 
$$
D_l:=\overline{\left\{[m]\in S_X\mid m\cap l\neq\emptyset,\ m\neq l\right\}}.
$$
Analogously if $[\overline l]\in S$ we can consider 

$$
D_{\overline{l}}:=\overline{\left\{[\overline m]\in S\mid \overline{m}\cap {\overline{l}}\neq\emptyset, \ \overline{m}\neq \overline{l}\right\}}.
$$

For any $[\overline l]\in S$ denote by $l'$ and $l''$ the two lines inside the double solid $Q$ such that $\pi_Q(l')=\pi_Q(l'')=\overline l$, that is, using notation of Proposition~\ref{liscia} we can write:
$$f^{-1}([\overline l])=\left\{[l'],[l'']\right\}.$$
It is known that
$$
f^{\star}D_{\overline{l}}=D_{l'}+D_{l''}
$$
(\emph{cf}. \cite[Formula 5.1, p. 61]{W}). Let $\{p,\, q\}=\overline l\cap X$. Note that the tangent plane to $X$ at the points $p$ intersects $X$ in a plane quartic curve with a node at $p$. By Riemann--Hurwitz formula, it admits six bitangents through $p$.
Let ${\overline{r}}_1,\,{\overline{r}}_2,\,{\overline{r}}_3, \,{\overline{r}}_4, \,{\overline{r}}_5$ be the other five bitangents passing through $p$ and ${\overline{s}}_1,\,{\overline{s}}_2,\,{\overline{s}}_3, \,{\overline{s}}_4, \,{\overline{s}}_5$ those through $q$. The following proposition holds.

\begin{prop} Let $\overline l$ be a general point of $S$. The curve $D_{\overline{l}}$ is irreducible and it has exactly $11$ nodes corresponding to $\overline l$,  ${\overline{r}}_1,\,{\overline{r}}_2,\,{\overline{r}}_3,\, {\overline{r}}_4, \,{\overline{r}}_5,\,{\overline{s}}_1,\,{\overline{s}}_2,\,{\overline{s}}_3, \,{\overline{s}}_4,$ and ${\overline{s}}_5$. The curve $D_{l'}$ is irreducible, it has only a node at the point $[l'']$. Moreover
\[D_{l'}\cap D_{l''}=\left\{ r_1',r_1'',\ldots,r_5',r_5'',s_1',s_1'',\ldots,s_5',s_5''\right\}.\]
For a general $[l]\in S_X$ , the geometric  genus of $D_l$ is $70$. For any $[l]\in S_X$, $D_l$ is ample and we have
$D_l^2=20$, $h^{0}(S_X\sO_{S_X}(D_l))=1$ and $h^{1}(S_X\sO_{S_X}(D_l))=71$ .
\end{prop}

\begin{proof}\cite[Section 3, Proposition 5.2, and Corollary 5.4]{W}.
\end{proof}

\subsection{Elliptic curves - Proof of Theorem~\ref{Th:C}}
We need now to recall that there is an identification between ${\rm{Alb}}(S_X)$ and ${\rm{Pic^{0}}}(S_X)$. Fix a point $[l_0]\in S_X$. We consider 
$\alpha([l_0])=0\in {\rm{Alb}}(S_X)$. 
We can define a morphism $$S_X\ni ([l])\longmapsto ([D_l]-[D_{l_0}]) \in {\rm{Pic^{0}}} (S_X).$$
Since $\alpha(S_X)$ generates ${\rm{Alb}}(S_X)$ as a group we may define the morphism
$\Phi_{[l_0]}\colon {\rm{Alb}}(S_X)\to {\rm{Pic^{0}}}(S_X)$ such that $\Phi_{[l_0]}\colon \alpha([l])\mapsto  [D_l]-[D_{l_0}]$, which is {\it{an isomorphism}}; see \cite[p. 301]{CG} or \cite[ (6.8) p. 77]{W}.

\smallskip

We are now ready to prove Theorem~\ref{Th:C}.
  
\begin{proof}[Proof of Theorem~\ref{Th:C}] First we show that if such a curve exists then it is smooth inside $S$. Indeed assume that $C\subset S$ is a curve such that $g(\widetilde C)=1$ where $\nu\colon \widetilde C\to C$ is its normalisation. 

By Proposition \ref{liscia} we can consider $\widetilde C':=\widetilde C\times_{S}S_X$. By Theorem \ref{collino} there is a non constant morphism $\widetilde C'\to {\rm{Alb}}(S_X)$. Since there are no rational curves inside an Abelian variety it follows that  $\widetilde C'$ is smooth. Since the covering $\widetilde C'\to\widetilde C$ is \'etale, by Theorem \ref{collino} it follows that  $\widetilde C'\hookrightarrow  {\rm{Alb}}(S_X)$. Then $\widetilde C'\to C$ is \'etale. This implies that $C$ is smooth.

Now consider again the $2$-to-$1$ \'etale cover $f\colon S_X\to S$ of Proposition \ref{liscia} and let $\overline E\subset S$ be an elliptic curve. Then for $f^{\star}(\overline E)$ two cases can occur:  $f^{\star}(\overline E)=E$ is a smooth elliptic curve and 
$f_{|E}\colon E\to \overline E$ is a $2$-to-$1$ \'etale covering or $f^{\star}(\overline E)=E_1+E_2$ where $f_{|E_{i}}\colon E_{i}\to \overline E$ is an isomorphism ($i=1,2$).
In both cases there exists an elliptic curve, which we denote $E$ since no confusion can arise, which is contained inside $S_X$ and such that the Albanese morphism $\alpha\colon S_X\to{\rm{Alb}}(S_X)$ restricted to $E$ is an embedding. We consider the case when 
$E=f^{\star}(\overline E)$ is an irreducible elliptic curve. The case where $f^{\star}(\overline E)$ is the union of two elliptic curves can be treated similarly.

We fix a point $[l_0]\in E$ and we consider the isomorphism $\Phi_{[l_0]}\colon {\rm{Alb}}(S_X)\to {\rm{Pic^{0}}}(S_X)$, which is induced by the morphism $\Phi_{[l_0]}\circ\alpha([l])=[D_{l}]- [D_{l_0}]$. Since $E\subset S_X$ and $E\hookrightarrow{\rm{Alb}}(S_X)$ 
this gives a surjective morphism
\[{\rm{Pic}}^{0}\left({\rm{Alb}}(S_X)\right)\simeq {\rm{Pic}}^{0}(S_X)\To  {\rm{Pic}}^{0}(E)\]
which induces the surjective morphism $h\colon S_X\to {\rm{Pic}}^{0}(E)$ given by $[l]\mapsto [D_{l\mid E}]- [D_{l_{0}\mid E}]$. 

 Notice that if $[l],[m]\in E$ then $D_{l}$ is not linearly equivalent to $D_{m}$ since $\Phi_{[l_{0}]|E}\colon E\to {\rm{Pic^{0}}}(S_X)$ is injective. 
 Let $\sigma$ be the involution of $S_X$ associated to the $2$-to-$1$ cover $S_X\to S$. Now, for every $[l]\in S_X$ we have $(\sigma^{\star}D_{l})_{|E}= \sigma^{\star}_{|E}(D_{l| E})$. This means that the morphism $h\colon S_X\to {\rm{Pic}}^{0}(E)$ is $\sigma$-equivariant. In particular its Stein factorisation $h=h_1\circ\tau$ is given by a fibration (that is a surjective morphism) $h_{1}\colon S_X\to E_1$ followed by an \'etale covering $\tau\colon E_1\to {\rm{Pic}}^{0}(E)$. By naturality even the morphism $h_{1}\colon S_X\to E_1$ is $\sigma$-equivariant. Since $E/\langle \sigma_{|E}\rangle=\overline E$ is an elliptic curve, the $\sigma$-equivariance of the Stein factorisation gives a non constant morphism $S\to {\rm{Pic}}^{0}(\overline E)$. If $f^{\star}(\overline E)=E_1+E_2$ then $E_1$ and $E_2$ are two copies of $\overline E$ and $\sigma$ is given by an $E$-automorphism without fixed points. Then we find two surjective morphisms $h_1\colon S_X\to  {\rm{Pic}}^{0}(E_1)$ and $h_2\colon S_X\to  {\rm{Pic}}^{0}(E_2)$ which are exchanged by $\sigma$. This implies again that we have a surjective morphism $S\to{\rm{Pic}}^0( {\overline E})$. This is a contradiction since $q(S)=0$ and it proves Theorem~\ref{Th:C}.
\end{proof}

\section{Proof of Theorem~\ref{Th:A}}

Let $X$ be a smooth quartic surface. From Theorem \ref{curel} we know that it contains an algebraic family of curves of geometric genus one; 
this family is parametrized by an algebraic variety, necessarily defined over the algebraic closure of $\kappa$. Pick one such curve,  say $C\subset X$, defined over the algebraic closure of $\kappa$, hence also over a finite extension of $\kappa$. After further enlarging this field if necessary, we can find a number field $\kappa'\supset\kappa$ such that $C$ is defined over $\kappa'$ and contains infinitely many rational points in $\kappa'$.

Let $p\in C(\kappa')$ be one such rational point. The tangent curve $X_p=X\cap T_p(X)$ has geometric genus $\leq 2$ and is singular at $p$. (For all but finitely many $p\in C(\kappa')$ its genus is exactly $2$, but this fact will not be used.) Since $X_p$ is a singular quartic with a node at $p$, one can define a morphism, over the field $\kappa'$, $X_p\to \mP^1$, as follows: identify $\mP^1$ with the pencil of lines in $T_p(X)$ passing through $p$ (note that the identification is defined over $\kappa'$) and send a generic point $q\in X_p$ to the line joining $p$ and $q$. This is a degree two morphism, identifying the quotient of $X_p$ by its hyperelliptic involution with the line $\mP^1$. Now, every rational point $x\in \mP^1(\kappa')$ has two quadratic (or rational) pre-images in $X_p$, hence the set of quadratic points in $X_p$ is infinite.

Moving $p$ in the infinite set $C(\kappa')$ of rational points on $C$, one obtains in this way a Zariski-dense set of quadratic points on $X$.\qed

\begin{rmk}
An alternative idea would be to use two genus one curves $C$ and $C'$; again after enlarging the ground field we can suppose that both curves possess infinitely many rational points. Any line joining a rational point on $C$ to a rational point on $C'$ is defined over the enlarged ground field and intersects the quartic in two more points which are quadratic (or rational). 
\end{rmk}

\section{Proof of Theorem~\ref{Th:B}}

Recall that the statement  amounts to proving the finiteness of the rational points on the surface $S$ of bitangents of $X$, over every number field. The proof  is achieved in two steps.

\begin{enumerate}[label=(\arabic*)]
\item By Proposition \ref{liscia} and by the Chevalley--Weil Theorem (see \emph{e.g.} \cite{CZ-Hilb}) proving the finiteness of rational points on $S$ over every number field amounts to proving the same result for the surface $S_X$. We shall then work on $S_X$ and fix a number field $\kappa'\supset\kappa$.
By Theorem~\ref{collino} we know that the Albanese map  $\alpha\colon S_X\rightarrow \mathrm{Alb}_{S_X}$ is a closed immersion.
 Then, by Faltings' theorem  (\emph{cf.} \cite[Theorem~1]{Fal}), we obtain at once that  all but finitely many  rational points of $S_X(\kappa')$ lie on the union of finitely many curves. We have thus proved the degeneracy of rational points on $S_X$, hence on $S$.
\item To end the proof of Theorem~\ref{Th:B}, we must pass from degeneracy to finiteness. Indeed we have to exclude the infinite families of rational points lying on finitely many curves on $S_X$. Again by Faltings' theorem, such curves must have geometric genus $\leq 1$; since the Albanese map $S_X\to \mathrm{Alb}_{S_X}$ has an injective differential, no curve is contracted by this map; now, since  Abelian varieties cannot contain any rational curve, the surface $S_X$ also does not contain curves of genus zero. The same holds for $S$, since the covering $f\colon S_X\to S$ is \'etale. It remains to exclude the presence  of  curves of genus $1$ on $S_X$ (again, this is the same as excluding such curves on $S$). This is exactly the content of Theorem~\ref{Th:C}.
\end{enumerate}


\end{document}